\documentclass[10pt]{amsart}
\usepackage{enumerate,amsmath,amssymb,latexsym,
amsfonts, amsthm, amscd, mathabx, url, hyperref}

\theoremstyle{plain}

\newtheorem{Theorem}{Theorem}[section]
\newtheorem{Lemma}[Theorem]{Lemma}
\newtheorem{Corollary}[Theorem]{Corollary}

\newtheorem{Definition}[Theorem]{Definition}

\newtheorem{Fact}[Theorem]{Fact}
\newtheorem{Question}[Theorem]{Question}
\numberwithin{equation}{section}
\newtheorem*{Claim}{Claim}

\newtheorem*{Claim1}{Claim 1}
\newtheorem*{Claim2}{Claim 2}

\numberwithin{equation}{section}

\addtocounter{section}{0}

\begin{document}

\title[Dominating reals and ordinal theories]{Dominating reals and ordinal theories}
\author[Cayley Robinson]{Cayley Robinson}

\address{Department of Mathematics, University of Florida\\
Gainesville, FL 32611, USA}
\email{chrobinson@ufl.edu}

 \keywords{}

\begin{abstract}
\begin{center} Hechler forcing changes the ordinal theory. \end{center}
\end{abstract}

\maketitle

\section{Introduction}
Let \(\mathbb{P}\) be a forcing. Must \(\mathbb{P}\) change the theory (i.e. change the truth of some sentence)? Must \(\mathbb{P}\) change the ordinal theory (i.e. change the truth of some formula whose only parameters are ordinals)? The answers depend on \(\mathbb{P}\) and on the universe with which we are dealing.

Suppose \(\mathbb{P}\) is either Cohen forcing or random forcing. Then in \(L\), \(\mathbb{P}\) changes the theory. On the other hand, in a \(\mathbb{P}\)-generic extension of \(L\), \(\mathbb{P}\) doesn't change the ordinal theory. Similarly, Sacks forcing can be iterated along inverted \(\omega\) \cite{Groszek}, and in the resulting generic extension, Sacks forcing doesn't change the ordinal theory.

Suppose instead that \(\mathbb{P}\) adds a dominating real (i.e. a real \(d \in \omega^{\omega}\) such that for each \(r \in \omega^{\omega}\) in the ground model, we have \(d(n) \geq r(n)\) for all but finitely many \(n \in \omega\)). In this case, Greg Hjorth showed (Fact \ref{Hjorth} below) that \(\mathbb{P}\) cannot be iterated along inverted \(\omega\). Must \(\mathbb{P}\) change the theory?

Under certain assumptions, the answer is no. As observed by Golshani and Mitchell \cite{GolshaniMitchell}, Woodin \cite{FriedmanWelchWoodin} used large cardinals to construct a transitive model of ZFC, satisfying \(V = L[r]\) for a real \(r\), whose theory is unchanged by any forcing whose generic object is a real.

In Woodin's model, collapse forcing (which adds a dominating real) doesn't change the theory, but does change the ordinal theory, because it changes the truth of the statement ``\(\alpha\) is uncountable" for some \(\alpha \in ON\). Must every forcing that adds a dominating real change the ordinal theory? We give three partial answers.

\begin{Theorem} \label{onepointone}
Let \(A\) and \(B\) be sets of ordinals. Suppose \(L[A] \subseteq L[B]\) and \(L[B]\) has a real that dominates the reals of \(L[A]\). Then \(L[A]\) and \(L[B]\) have different ordinal theories.
\end{Theorem}

In particular, in Woodin's model, the ordinal theory is changed by any forcing that adds a dominating real.

\begin{Theorem} \label{onepointtwo}
Suppose there is an extendible cardinal. Let \(\mathbb{P}\) be a forcing that adds a dominating real. Then \(\mathbb{P}\) changes the ordinal theory.
\end{Theorem}

\begin{Theorem} \label{onepointthree}
Let \(\mathbb{P}\) be a definable cone homogeneous forcing that adds a dominating real. Then \(\mathbb{P}\) changes the ordinal theory.
\end{Theorem}

In particular, Hechler forcing changes the ordinal theory. Using the terminology of Hamkins and Löwe \cite{HamkinsLöwe}: in each model of ZFC, the modal logic of Hechler forcing with ordinal parameters is nontrivial.

The paper is organized as follows. Section 2 has many preliminaries, including a construction of \(M[x]\) and a proof of Hjorth's result. Section 3 covers a general theorem (Theorem \ref{MTheorem}) that has Theorems \ref{onepointone} and \ref{onepointtwo} (Corollaries \ref{LTheorem} and \ref{MantleTheorem}) as consequences. Section 4 covers Theorem \ref{onepointthree} (Theorem \ref{GenTheorem}). Section 5 states some open problems.

\section{Preliminaries}
The reader may want to skip this section, then refer to it later as needed. We mostly use notation from Kunen's textbook \cite{Kunen}.

\begin{Definition} Let \(V\) be an inner model of \(W\). \(V\) and \(W\) {\normalfont have the same ordinal theory} (written \(V \equiv_{ON} W\)) if they agree on all formulas whose only parameters are ordinals.
\end{Definition}

\begin{Definition} Let \(d \in \omega^{\omega}\). Let \(A\) be a class. If \(d\) dominates every real \(r \in \omega^{\omega} \cap A\), then we say that \(d\) {\normalfont dominates} \(A\).
\end{Definition}

\begin{Definition} A forcing \(\mathbb{P}\) is {\normalfont cone homogeneous} if for each \(p, q \in \mathbb{P}\), there exist \(p' \leq p\) and \(q' \leq q\) such that \(\mathbb{P}\restriction{p'}\) is isomorphic to \(\mathbb{P}\restriction{q'}\).
\end{Definition}

Lemma \ref{PreservesLemma} exploits the following property of cone homogeneous forcings: if \(\mathbb{P}\) is cone homogeneous, then for each formula \(\varphi(x_1, ..., x_n)\) and each \(a_1, ..., a_n\), we have either \(1_{\mathbb{P}} \Vdash \varphi(\check{a}_1, ..., \check{a}_n)\) or \(1_{\mathbb{P}} \Vdash \neg\varphi(\check{a}_1, ..., \check{a}_n)\).

\begin{Definition}
{\normalfont Hechler forcing} is the forcing \(\{\langle s, f \rangle : s \in \omega^{< \omega} \land f \in \omega^{\omega}\}\), where \(\langle s', f' \rangle \leq \langle s, f \rangle\) if and only if \(s' \supseteq s\) and \(\forall n \in \omega \thinspace [f'(n) \geq f(n)]\) and \(\forall n \in \text{dom}(s') \setminus \text{dom}(s) \thinspace [s'(n) \geq f(n)]\).
\end{Definition}

Hechler forcing is cone homogeneous, and adds a dominating real.

Part (7) of Fact \ref{MDef} uses the following result, which is II.4.27 in \cite{Kunen}.

\begin{Fact} \label{KunenTheorem}
Let \(M\) be a transitive class. Suppose \(\forall b \subseteq M \thinspace \exists c \in M \thinspace [b \subseteq c]\) and \(M \models \text{Comprehension}\). Then \(M \models \text{ZF}\).
\end{Fact}

The following definition of \(M[x]\) is adapted from the ``Constructible sets" section of \cite{Monk}, which is itself adapted from \cite{Jech}.

\begin{Definition}
For each set \(S\) and each \(b \subseteq S\), we say that \(b\) is {\normalfont definable over} \((S, \in)\) if there exist a formula \(\varphi\) and \(a_1, ..., a_n \in S\) such that \(b = \{a \in S : (S, \in) \models \varphi(a, a_1, ..., a_n)\}\). \newline

\noindent For each set \(S\), let \(\text{def}(S) := \{b \subseteq S : b \text{ is definable over } (S, \in)\}\). \newline

\noindent Suppose \(M\) is an inner model of ZF, and \(x\) is a subset of \(M\). \newline
Define \(\langle L_{\alpha}[M, x] : \alpha \in ON \rangle\) by recursion: \newline
\(L_0[M, x] := \emptyset\); \newline
\(L_{\alpha+1}[M, x] := (V_{\alpha} \cap (M \cup \{x\})) \cup \text{def}(L_{\alpha}[M, x])\); \newline
and for limit \(\gamma\), \(L_{\gamma}[M, x] := \bigcup_{\alpha < \gamma} L_{\alpha}[M, x]\). \newline
Finally, let \(M[x] := \bigcup_{\alpha \in ON} L_{\alpha}[M, x]\).
\end{Definition}

We can find no detailed reference for the next two results, so we prove both of them.

\begin{Fact} \label{MDef}
Let \(M\) be an inner model of ZF. Let \(x \subseteq M\). Then: \newline
\noindent (1) \(\forall \alpha \in ON \thinspace [L_{\alpha}[M, x] \subseteq V_{\alpha}]\). \newline
\noindent (2) \(\forall \beta \in ON \thinspace \forall \alpha \leq \beta \thinspace [L_{\alpha}[M, x] \subseteq L_{\beta}[M, x]]\). \newline
\noindent (3) \(\forall \beta \in ON \thinspace [L_{\beta}[M, x] \text{ is transitive}]\) (so \(M[x]\) is transitive). \newline
\noindent (4) \(\forall b \subseteq M[x] \thinspace \exists c \in M[x] \thinspace [b \subseteq c]\). \newline
\noindent (5) \(M[x] \models \text{Comprehension}\). \newline
\noindent (6) \(M \cup \{x\} \subseteq M[x]\). \newline
\noindent (7) \(M[x]\) is an inner model of ZF. \newline
\noindent (8) If \(N\) is an inner model of ZF, and \(M \cup \{x\} \subseteq N\), then \(M[x] \subseteq N\).
\end{Fact}

\begin{proof}
(1) holds by induction on \(\alpha\). \newline

\noindent We prove (2) and (3) simultaneously by induction on \(\beta\). The base case and the limit case are clear. For the successor case, let \(\beta \in ON\), and suppose (2) and (3) hold for \(\beta\). Then \(L_{\beta}[M, x] \subseteq L_{\beta + 1}[M, x]\) as desired, because for each \(b \in L_{\beta}[M, x]\), we have \newline
\(b = \{a \in L_{\beta}[M, x] : a \in b\} \in \text{def}(L_{\beta}[M, x])\). To show that \(L_{\beta + 1}[M, x]\) is transitive, let \newline
\(b \in L_{\beta + 1}[M, x]\). If \(b \in \text{def}(L_{\beta}[M, x])\), then \(b \subseteq L_{\beta}[M, x] \subseteq L_{\beta + 1}[M, x]\). Otherwise, we must have \(b \in V_{\beta} \cap (M \cup \{x\})\), so \(b \subseteq V_{\beta} \cap (M \cup \{x\}) \subseteq L_{\beta + 1}[M, x]\). \newline

\noindent For (4), let \(b \subseteq M[x]\). For each \(a \in b\), let \(\xi(a)\) be least with \(a \in L_{\xi(a)}[M, x]\). \newline
Let \(\eta := \text{sup}\{\xi(a) + 1 : a \in b\}\) and \(c := L_{\eta}[M, x]\). \newline
Then we have \(b \subseteq c\) and \(c \in \text{def}(L_{\eta}[M, x]) \subseteq L_{\eta + 1}[M, x] \subseteq M[x]\). \newline

\noindent For (5), let \(\varphi(a, c, v_1, ..., v_n)\) be a formula. We must show \newline
\(\forall c, v_1, ..., v_n \in M[x] \thinspace \exists b \in M[x] \thinspace \forall a \in M[x] \thinspace [a \in b \leftrightarrow (a \in c \land M[x] \models \varphi(a, c, \vec{v}))]\). \newline
Let \(c, v_1, ..., v_n \in M[x]\). Take \(\xi \in ON\) with \(c, v_1, ..., v_n \in L_{\xi}[M, x]\). By Lévy Reflection (since \(\langle L_{\alpha}[M, x]: \alpha \in ON \rangle\) is continuous and increasing), take \(\eta > \xi\) so that for each \(a \in L_{\eta}[M, x]\), the models \(L_{\eta}[M, x]\) and \(M[x]\) agree on \(\varphi(a, c, \vec{v})\). Let \newline
\(b := \{a \in L_{\eta}[M, x] : a \in c \thinspace \land \thinspace L_{\eta}[M, x] \models \varphi(a, c, \vec{v})\}\). \newline
Then \(b \in L_{\eta + 1}[M, x] \subseteq M[x]\) and \(\forall a \in M[x] \thinspace [a \in b \leftrightarrow (a \in c \land M[x] \models \varphi(a, c, \vec{v}))]\). \newline

\noindent (6) follows from the definition of \(L_{\alpha + 1}[M, x]\). So \(ON \subseteq M[x]\), so (by (3), (4), (5), and Fact \ref{KunenTheorem}) we have (7). \newline

\noindent For (8), suppose \(N\) is an inner model of ZF, with \(M \cup \{x\} \subseteq N\). Let \(\xi \in ON\) be an arbitrary limit. Then \(V_{\xi}^M \cup \{x\} \in N\), and for each \(\alpha < \xi\), \(L_{\alpha + 1}[M, x] = (V_{\alpha}^N \cap (V_{\xi}^M \cup \{x\})) \cup \text{def}(L_{\alpha}[M, x])\). So by the absoluteness of ``def", we have \(\langle L_{\alpha}[M, x] : \alpha \leq \xi \rangle \in N\).
\end{proof}

Theorem \ref{MTheorem} uses the following result, which says that certain well-orders in \(M[x]\) are definable from well-orders in \(M\).

\begin{Fact} \label{KTheorem}
There is a class function \(K\) such that for each inner model \(M\) of ZF, each limit \(\xi \in ON\), each well-order \(\triangleleft\) of \(V_{\xi}^M\), and each \(x \subseteq M\): \newline
\noindent (1) \(K(\triangleleft, x)\) is a well-order of \(L_{\xi}[M, x]\), and \newline
\noindent (2) for each transitive \(N \models\) ZF with \(\triangleleft, x \in N\), we have \(K^N(\triangleleft, x) = K(\triangleleft, x)\). \newline
In particular, if \(M\) is an inner model of ZFC and \(x \subseteq M\), then \(M[x] \models\) AC.
\end{Fact}

\begin{proof}
Let \(M\) be an inner model of ZF. Let \(\xi \in ON\) be a limit. Let \(\triangleleft\) be a well-order of \(V_{\xi}^M\). Let \(x \subseteq M\). Note that \(\xi\) and \(V_{\xi}^M\) are definable from \(\triangleleft\). (So for each model \(N\) as in (2), we have \(\forall \alpha \leq \xi \thinspace [V_{\alpha}^M \subseteq V_{\alpha}^N]\).) \newline

\noindent Using \(\triangleleft\) and \(x\), define \(\langle S_{\alpha}: \alpha \leq \xi \rangle\) by recursion: \newline
\(S_0 := \emptyset\); \newline
\(S_{\alpha + 1} := (V_{\alpha} \cap (V_{\xi}^M \cup \{x\})) \cup \text{def}(S_{\alpha})\); \newline
and for limit \(\gamma\), \(S_{\gamma} := \bigcup_{\alpha < \gamma} S_{\alpha}\). \newline
(So for each model \(N\) as in (2), we have \(\forall \alpha \leq \xi \thinspace [S_{\alpha}^N = L_{\alpha}[M, x]]\).) \newline

\noindent Define \(\langle <_{\alpha} \thinspace : \alpha \leq \xi \rangle\) by recursion so that each \(<_{\alpha}\) well-orders \(S_{\alpha}\), as follows. Let \(<_0 \thinspace := \emptyset\). \newline
Assume that \(<_{\alpha}\) has been defined. Using first \(<_{\alpha}\), then a canonical well-order of all formulas: well-order \(S_{\alpha}^{<\omega}\), then well-order \(\text{def}(S_{\alpha})\). Define a well-order \(<_{\alpha+1}\) of \(S_{\alpha+1}\) by having the elements of \(S_{\alpha}\) come first (ordered by \(<_{\alpha}\)), followed by the new elements of \(\text{def}(S_{\alpha})\), followed by the new elements of \(V_{\alpha} \cap V_{\xi}^M\) (ordered by \(\triangleleft\)), followed by (if \(x \in V_{\alpha}\) and \(x\) is a new element) \(x\). \newline
For limit \(\gamma\), let \(<_{\gamma} \thinspace := \bigcup_{\alpha < \gamma} <_{\alpha}\). \newline

\noindent Finally, let \(K(\triangleleft, x) := \thickspace <_{\xi}\).
\end{proof}

The following unpublished result of Greg Hjorth will be crucial. We include a proof, taken from \cite{Brendle}.

\begin{Fact} \label{Hjorth}
There is no sequence \(\langle f_n : n \in \omega \rangle\) of elements of \(\omega^{\omega}\) such that \newline
\(\forall m \in \omega \thinspace [f_m \text{ dominates } L[\langle f_n : n > m \rangle]]\).
\end{Fact}

\begin{proof}
Let \(\pi : \omega \times \omega \rightarrow \omega\) be the Cantor pairing function. \newline
For each \(x \in \omega^{\omega}\), let \(\sqsubset_x \thinspace := \{(m, n) \in \omega \times \omega: x(\pi(m,n)) > 0\}\), \newline and let \(T_x := \{s \in \omega^{<\omega} : \forall i, j < \lvert s \rvert \thinspace [i < j \rightarrow s(j) \sqsubset_x s(i)]\}\) be the corresponding tree. \newline
For each \(n \in \omega\) and \(f,x \in \omega^{\omega}\), let \(T_x^n(f) := \{s \in T_x : \forall i < \lvert s \rvert \thinspace [i \geq n \rightarrow s(i) \leq f(i)]\}\). \newline

\noindent Let \(WO := \{x \in \omega^{\omega} : \thickspace \sqsubset_x \text{is a well-order of some nonempty subset of } \omega\}\). \newline
Then \(\forall x \in \omega^{\omega} \thinspace [T_x \text{ is well-founded} \leftrightarrow x \in WO]\). For each \(f \in \omega^{\omega}\), let \newline
\(\omega_1^{CK}(f) := \text{sup}\{\text{type}(\sqsubset_x) : x \in WO \text{ and } x \text{ is recursive in } f\}\). \newline

\noindent The effective version of the Boundedness Theorem (see e.g. 4A.4 in \cite{Moschovakis} for a proof) says that for each \(P \subseteq \omega^{\omega}\) and each \(\Delta_1^1\) function \(h: \omega^{\omega} \rightarrow \omega^{\omega}\) satisfying \(\forall x \in \omega^{\omega} \thinspace [x \in P \leftrightarrow h(x) \in WO]\): the set \(P\) is \(\Delta_1^1\) if and only if \(\text{sup}\{\text{type}(\sqsubset_{h(x)}) : x \in P\} < \omega_1^{CK}\).

\begin{Claim1}
Let \(f, g, x \in \omega^{\omega}\). Suppose \(f\) dominates \(L[g]\), and \(x\) is recursive in \(g\). Then \newline
\(x \in WO \leftrightarrow \forall n \in \omega \thinspace [T_x^n(f) \text{ is well-founded}]\).
\end{Claim1}

\noindent \textit{Proof of Claim 1.} If \(n \in \omega\) and \(T_x^n(f)\) is ill-founded, then \(T_x\) is ill-founded, so \(x \notin WO\). \newline

\noindent Conversely, suppose \(x \notin WO\). Then \(T_x\) is ill-founded. By absoluteness, take an infinite branch \(y \in \omega^{\omega} \cap L[g]\) through \(T_x\). Then \(f\) dominates \(y\). Take \(n \in \omega\) so that \(\forall i \geq n \thinspace [y(i) \leq f(i)]\). Then \(y\) is an infinite branch through \(T_x^n(f)\). \hfill \qedsymbol

\begin{Claim2}
Let \(f, g \in \omega^{\omega}\). Suppose \(f\) dominates \(L[g]\). Then \(\omega_1^{CK}(f, g) > \omega_1^{CK}(g)\).
\end{Claim2}

\noindent \textit{Proof of Claim 2.} Define \(h : \omega^{\omega} \rightarrow \omega^{\omega}\) by \(h(x) = \begin{cases} 
x & \text{if } x \text{ is recursive in } g \\
0 & \text{otherwise.}
\end{cases}
\) \newline

\noindent Let \(P := \{x \in \omega^{\omega} : h(x) \in WO\}\). Then for each \(x \in \omega^{\omega}\), we have \newline
\(x \in P \leftrightarrow h(x) \in WO \leftrightarrow x \text{ is recursive in } g \text{ and } x \in WO\) \newline
\(\leftrightarrow\) (by Claim 1) \(x \text{ is recursive in } g \text{ and } \forall n \in \omega \thinspace [T_x^n(f) \text{ is well-founded}]\) \newline
\(\leftrightarrow\) (by Kőnig's Lemma) \(x \text{ is recursive in } g \text{ and } \forall n \in \omega \thinspace [T_x^n(f) \text{ is finite}]\). \newline

\noindent So \(h\) is \(\Delta_1^1(g)\) and \(P\) is \(\Delta_1^1(f, g)\). \newline
So \(\omega_1^{CK}(f, g) >\) (by the relativized effective Boundedness Theorem) \(\text{sup}\{\text{type}(\sqsubset_{h(x)}) : x \in P\}\) \newline
\(= \text{sup}\{\text{type}(\sqsubset_{h(x)}) : x \in WO \text{ and } x \text{ is recursive in } g\}\) \newline
\(= \text{sup}\{\text{type}(\sqsubset_x) : x \in WO \text{ and } x \text{ is recursive in } g\} = \omega_1^{CK}(g)\). \hfill \qedsymbol \newline

\noindent Finally, toward a contradiction, suppose \(\langle f_n : n \in \omega \rangle\) is a sequence of elements of \(\omega^{\omega}\), and \(\forall m \in \omega \thinspace [f_m \text{ dominates } L[\langle f_n : n > m \rangle]]\). Then by Claim 2, the sequence \newline
\(\langle \omega_1^{CK}(\langle f_n : n > m \rangle) : m \in \omega \rangle\) is strictly decreasing.
\end{proof}

Lemma \ref{PreservesLemma} uses the ground model definability theorem, which was proved independently by Laver \cite{Laver} and Woodin \cite{Woodin}:

\begin{Fact} \label{LaverWoodin}
There is a formula \(\varphi(x, y)\) such that for each model \(M\) of ZFC, each forcing \(\mathbb{P} \in M\), and each filter \(G \subseteq \mathbb{P}\) that is \(\mathbb{P}\)-generic over \(M\), we have \newline
\(\forall a \in M[G] \thinspace [a \in M \leftrightarrow M[G] \models \varphi(a, \mathcal{P}(\lvert \mathbb{P} \rvert)^M)]\).
\end{Fact}

By refining the ground model definability theorem, Fuchs, Hamkins, and Reitz \cite{FuchsHamkinsReitz} showed that the \textit{mantle} is a definable class:

\begin{Definition} An inner model \(M\) of ZFC is a {\normalfont ground} if for some forcing \(\mathbb{P} \in M\) and some filter \(G \subseteq \mathbb{P}\): \(G\) is \(\mathbb{P}\)-generic over \(M\), and \(M[G] = V\). The {\normalfont mantle} is the intersection of all grounds.
\end{Definition}

Usuba showed that the mantle is a forcing invariant inner model of ZFC \cite{Usuba1}, and that in the presence of an \textit{extendible} cardinal, the mantle is a ground \cite{Usuba2}:

\begin{Definition} A cardinal \(\kappa > \aleph_0\) is {\normalfont extendible} if for each \(\alpha > \kappa\), there exist an ordinal \(\beta\) and an elementary embedding \(j : V_{\alpha} \rightarrow V_{\beta}\) with critical point \(\kappa\).
\end{Definition}

\begin{Fact} \label{Usuba}
The mantle is forcing invariant (i.e. if \(V[G]\) is any generic extension of \(V\), then \(V\) and \(V[G]\) have the same mantle), and the mantle is an inner model of ZFC. If there is an extendible cardinal, then the mantle is a ground.
\end{Fact}

Finally, Theorem \ref{GenTheorem} uses the following result of Shelah \cite{Shelah}. (See \cite{Cummings} for a short proof.)

\begin{Fact} \label{Shelah}
Let \(\kappa\) be a singular strong limit cardinal with \(\text{cf}(\kappa) > \omega\). Then there exists \(x \subseteq \kappa\) such that \(\mathcal{P}(\kappa) \subseteq HOD_x\) (where \(HOD_x\) denotes the class of all sets that are hereditarily definable from ordinals and \(x\)).
\end{Fact}

\section{Models of the form M[A]}
We will prove a general theorem, then give two applications.

\begin{Theorem} \label{MTheorem}
Let \(M\) be a definable inner model of ZFC. Let \(A\) and \(B\) be sets of ordinals. Suppose \(M[A] \subseteq M[B]\), \(M^{M[A]} = M = M^{M[B]}\), and \(M[B]\) has a real that dominates \(M[A]\). Then \(M[A] \not\equiv_{ON} M[B]\).
\end{Theorem}

\begin{proof}
Take \(\kappa \geq \aleph_0\) with \(B \subseteq \kappa\). In \(M[A]\) and in \(M[B]\), define a sequence \(\langle R_{\alpha} : \alpha \in ON \rangle\) of binary relations on \(\mathcal{P}(\kappa)\) by recursion: \newline
\(xR_0y\) iff \(y \in M[x]\); \newline
\(xR_{\alpha + 1}y\) iff \(xR_{\alpha}y\), \(y \in \text{dom}(R_{\alpha})\), and \(\exists d \in \omega^{\omega} \cap M[x] \thinspace [d \text{ dominates } M[y]]\); \newline
and for limit \(\gamma\), \(xR_{\gamma}y\) iff \(\forall \alpha < \gamma \thinspace [xR_{\alpha}y]\). \newline

\noindent In \(M[A]\) and in \(M[B]\): \(\langle R_{\alpha} : \alpha \in ON \rangle\) is decreasing, so it eventually stabilizes by Replacement. Write \(R_{\infty}\) for the limit relation.

\begin{Claim}
In \(M[A]\) and in \(M[B]\): \(R_{\infty} = \emptyset\).
\end{Claim}

\noindent \textit{Proof of Claim.} Take the class function \(K\) from Fact \ref{KTheorem}. For each \(x \subseteq \kappa\), let \(\alpha_x \in ON\) be least so that \((\mathcal{P}(\kappa) \times \omega^{\omega}) \cap M[x] \subseteq L_{\alpha_x}[M, x]\). Let \(\xi\) be the first limit ordinal above \(\text{sup}\{\alpha_x : x \subseteq \kappa\}\). Take a well-order \(\triangleleft \in M\) of \(V_{\xi}^M\). \newline

\noindent Then for each \(x \subseteq \kappa\): \(K(\triangleleft, x)\) is absolute for transitive models of ZF, and \(K(\triangleleft, x)\) is a well-order of \(L_{\xi}[M, x]\), so (by our choice of \(\xi\)) \(K(\triangleleft, x)\) well-orders \((\mathcal{P}(\kappa) \times \omega^{\omega}) \cap M[x] \subseteq L_{\xi}[M, x]\). \newline

\noindent In particular, \((\mathcal{P}(\kappa) \times \omega^{\omega}) \cap M[x]\) is definable from \(\kappa, \triangleleft \in M\) and \(x \subseteq \kappa\). \newline
So in each model \(M[z]\), \(\langle R_{\alpha} \cap M[z] : \alpha \in ON \rangle\) is definable from \(\kappa, \triangleleft \in M\). \newline

\noindent Also, \(\forall x \in \text{dom}(R_{\infty}) \thinspace \exists y \in \text{dom}(R_{\infty}) \cap M[x] \thinspace \exists d \in \omega^{\omega} \cap M[x] \thinspace [d \text{ dominates } M[y]]\). \newline
So \(\forall x \in \text{dom}(R_{\infty}) \thinspace [M[x] \models ``\exists y \in \text{dom}(R_{\infty}) \thinspace \exists d \in \omega^{\omega} \thinspace [d \text{ dominates } M[y]]"]\), \newline
where the ``\(R_{\infty}\)" and the ``\(M[y]\)" after the \(\models\) are defined from \(\kappa, \triangleleft \in M\). \newline

\noindent Toward a contradiction, assume \(R_{\infty} \neq \emptyset\). Take \(x_0 \in \text{dom}(R_{\infty})\). \newline

\noindent Define \(\langle x_n, d_n : n \in \omega \rangle\) by recursion: given \(x_n \in \text{dom}(R_{\infty})\), \newline
let \((x_{n + 1}, d_n) \in (\mathcal{P}(\kappa) \times \omega^{\omega}) \cap M[x_n]\) be the \(K(\triangleleft, x_n)\)-least pair such that \newline
\(M[x_n] \models ``x_{n+1} \in \text{dom}(R_{\infty}) \text{ and } d_n \text{ dominates } M[x_{n + 1}]"\), \newline
where the ``\(R_{\infty}\)" and the ``\(M[x_{n + 1}]\)" after the \(\models\) are defined from \(\kappa, \triangleleft \in M\). \newline

\noindent In the end, \(\forall n \in \omega \thinspace [M[x_n] \supseteq M[x_{n+1}]]\); \(\forall n \in \omega \thinspace [d_n \text{ dominates } M[x_{n+1}]]\); \newline
and \(\forall m \in \omega \thinspace [\langle x_n, d_n : n \geq m \rangle \in M[x_m]]\). This contradicts Fact \ref{Hjorth}. \hfill \qedsymbol \newline

\noindent Let \(x := B\). By induction, we have \(\forall \alpha \in ON \thinspace [R_{\alpha}^{M[A]} \subseteq R_{\alpha}^{M[B]}]\) and \newline
\(\forall \alpha \in ON \thinspace \forall y \in \text{dom}(R_{\alpha}^{M[A]}) \thinspace [xR_{\alpha}^{M[B]}y]\). \newline

\noindent Take \(d \in \omega^{\omega} \cap M[B]\) so that \(d\) dominates \(M[A]\). Toward a contradiction, assume \newline
\(M[A] \equiv_{ON} M[B]\). We will contradict the claim by showing \(\forall \alpha \in ON \thinspace [x \in \text{dom}(R_{\alpha}^{M[B]})]\). \newline

\noindent The base case and the limit case are clear. \newline
For the successor case, let \(\alpha \in ON\), and suppose \(x \in \text{dom}(R_{\alpha}^{M[B]})\). Since \(M[A] \equiv_{ON} M[B]\), take \(y \in \text{dom}(R_{\alpha}^{M[A]})\). Then \(xR_{\alpha}^{M[B]}y\); \(y \in \text{dom}(R_{\alpha}^{M[B]})\); \(d \in \omega^{\omega} \cap M[x]\); and \(d \text{ dominates } \newline
M[y] \subseteq M[A]\). So \(xR_{\alpha + 1}^{M[B]}y\), so \(x \in \text{dom}(R_{\alpha + 1}^{M[B]})\).
\end{proof}

Our proof of Theorem \ref{MTheorem} actually shows that \(M[A]\) and \(M[B]\) disagree on the first \(\alpha\) satisfying \(R_{\alpha} = \emptyset\); more specifically, the first such \(\alpha\) in \(M[A]\) is less than the first such \(\alpha\) in \(M[B]\).

\begin{Corollary} \label{LTheorem}
Let \(A\) and \(B\) be sets of ordinals. Suppose \(L[A] \subseteq L[B]\) and \(L[B]\) has a real that dominates \(L[A]\). Then \(L[A] \not\equiv_{ON} L[B]\).
\end{Corollary}

\begin{proof}
Since \(L^{L[A]} = L = L^{L[B]}\), this follows directly from Theorem \ref{MTheorem}.
\end{proof}

In fact, by the remark before Corollary \ref{LTheorem}: for each \(\kappa \geq \aleph_0\), the tree \newline
\(T(\kappa) := \{t \in \mathcal{P}(\kappa)^{<\omega} : \text{each } L[t(n)] \text{ is a dominating extension of } L[t(n + 1)]\}\) is well-founded. (If there were an infinite branch \(b \in \mathcal{P}(\kappa)^{\omega}\) through \(T(\kappa)\), then letting \(\alpha_n\) denote the least ordinal such that \(L[b(n)] \models ``R_{\alpha_n} = \emptyset"\), the sequence \(\langle \alpha_n : n \in \omega \rangle\) would be strictly decreasing.)

\begin{Corollary} \label{MantleTheorem}
Suppose there is an extendible cardinal. Let \(\mathbb{P}\) be a forcing that adds a dominating real. Then \(\mathbb{P}\) changes the ordinal theory.
\end{Corollary}

\begin{proof}
Let \(M\) denote the mantle. By assumption and Fact \ref{Usuba}, take a set \(G\) with \(M[G] = V\). Code \(G\) as a set of ordinals, in the following standard way. \newline

\noindent Take a cardinal \(\lambda\) and a bijection \(f : \text{trcl}(\{G\}) \rightarrow \lambda\). \newline
Let \(E \subseteq \lambda \times \lambda\) be the unique relation such that \(f : \langle \text{trcl}(\{G\}), \in \rangle \cong \langle \lambda, E \rangle\). \newline
Take an ordinal pairing function \(g : \lambda \times \lambda \rightarrow \lambda\). Finally, let \(A := g''E \subseteq \lambda\). \newline

\noindent Then we can recover \(G\) from \(A\) as follows: we recover first \(\langle \lambda, E \rangle\), then (using the transitive collapse) \(\text{trcl}(\{G\})\), then \(G\). So \(M[G] = V = M[A]\). \newline

\noindent Let \(H\) be \(\mathbb{P}\)-generic over \(V\). Using the same method as above, code \((G, H)\) as a set \(B\) of ordinals. Then \(V[H] = M[B]\). Finally, \(M[A] \subseteq M[B]\); by Fact \ref{Usuba}, \(M\) is the mantle of both \(M[A]\) and \(M[B]\); and \(M[B]\) has a real that dominates \(M[A]\). So by Theorem \ref{MTheorem}, \(M[A] \not\equiv_{ON} M[B]\).
\end{proof}

\section{Definable cone homogeneous forcings}
We will show that the ordinal theory is changed by any definable cone homogeneous forcing that adds a dominating real. Since we plan to apply Fact \ref{Shelah}, we work with models of the form \(HOD_x\).

To deal with the non-absoluteness of \(HOD_x\), we use the iteration \(\mathbb{P}^n\). Informally, \(\mathbb{P}^n\) denotes \(\mathbb{P} \ast ... \ast \mathbb{P}\) (with \(n\) \(\mathbb{P}\)'s). Formally, \(\mathbb{P}^n\) is defined for each \(n \in \omega\) as follows.

\begin{Definition} Let \(\mathbb{P}\) be a definable cone homogeneous forcing that doesn't change the ordinal theory. Let \(F\) denote the usual surjection from \(ON\) to \(OD\). Let \(\xi \in ON\) be least with \(\mathbb{P} = F(\xi)\). Let \(\mathbb{P}^0\) be trivial. Let \(\mathbb{P}^1 := \mathbb{P}\). \newline

\noindent Let \(n \geq 1\), and suppose \(\mathbb{P}^n\) has been defined. Then let \newline
\(\mathring{\mathbb{Q}}_n := \bigcup \{\sigma : \sigma \text{ is a } \mathbb{P}^n \text{-name of minimum rank such that } 1_{\mathbb{P}^n} \Vdash ``\sigma = F(\check{\xi})"\}\) \newline
(so we have \(1_{\mathbb{P}^n} \Vdash ``\mathring{\mathbb{Q}}_n = F(\check{\xi})"\)); let \(\mathring{\mathbb{R}}_n\) be the trivial \(\mathbb{P}^n\)-name for a forcing; \newline
and let \(\mathbb{P}^{n + 1} := \begin{cases} 
\mathbb{P}^n \ast \mathring{\mathbb{Q}}_n & \text{if } 1_{\mathbb{P}^n} \Vdash ``F(\check{\xi}) \text{ is a cone homogeneous forcing}" \\
\mathbb{P}^n \ast \mathring{\mathbb{R}}_n & \text{otherwise.}
\end{cases}
\)
\end{Definition}

Note that \(\forall n \in \omega \thinspace [1_{\mathbb{P}^n} \Vdash ``F(\check{\xi}) \text{ is a cone homogeneous forcing}"]\). \newline
(Otherwise, \(V\) and its \(\mathbb{P}\)-generic extension would disagree on the first \(n \in \omega\) such that \newline \(1_{\mathbb{P}^n} \nVdash ``F(\check{\xi}) \text{ is a cone homogeneous forcing}"\), so \(\mathbb{P}\) would change the ordinal theory.)

So \(\forall n \geq 1 \thinspace [\mathbb{P}^{n + 1} = \mathbb{P}^n \ast \mathring{\mathbb{Q}}_n]\), and the trivial names \(\mathring{\mathbb{R}}_n\) are never used.

\begin{Definition} Let \(\mathbb{P}\) be a forcing. Let \(x\) be a set of ordinals. \newline
If \(1_{\mathbb{P}} \Vdash ``HOD_x^V = HOD_x"\), then we say that \(\mathbb{P}\) {\normalfont preserves} \(HOD_x\).
\end{Definition}

\begin{Lemma} \label{PreservesLemma} Let \(\mathbb{P}\) be a definable cone homogeneous forcing that doesn't change the ordinal theory. Let \(a := \langle \mathcal{P}(\lvert \mathbb{P}^n \rvert) : n \in \omega \rangle\). Let \(x\) be a set of ordinals. Suppose \newline
\(\forall n \in \omega \thinspace [1_{\mathbb{P}^n} \Vdash ``\check{a} \in OD_x"]\). Then \(\forall n \in \omega \thinspace [\mathbb{P}^n \text{ preserves } HOD_x]\).
\end{Lemma}

\begin{proof}
By induction on \(n \in \omega\). Let \(K\) be \(\mathbb{P}^{n+1}\)-generic over \(V\). Take \(G\) and \(H\) so that \(G\) is \(\mathbb{P}^n\)-generic over \(V\); \(H\) is \(\mathbb{P}\)-generic over \(V[G]\) (where \(``\mathbb{P}"\) is interpreted in \(V[G]\)); and \(V[G][H] = V[K]\). \newline

\noindent Then \(HOD_x^V =\) (by induction) \(HOD_x^{V[G]} \supseteq\) (since \(V[G]\) satisfies ``\(\mathbb{P}\) is cone homogeneous") \(HOD_x^{V[G][H]} = HOD_x^{V[K]}\). Conversely, \(a\) is definable in \(V[K]\) from ordinals and \(x\), so by Fact \ref{LaverWoodin}, \(V\) is definable in \(V[K]\) from ordinals and \(x\); so \(HOD_x^V \subseteq HOD_x^{V[K]}\).
\end{proof}

\begin{Theorem} \label{GenTheorem}
Let \(\mathbb{P}\) be a definable cone homogeneous forcing that adds a dominating real. Then \(\mathbb{P}\) changes the ordinal theory.
\end{Theorem}

\begin{proof}
Toward a contradiction, assume that \(\mathbb{P}\) doesn't change the ordinal theory. Let \(G\) be \(\mathbb{P}\)-generic over \(V\). \newline

\noindent In \(V[G]\): Code \(a := \langle \mathcal{P}(\lvert \mathbb{P}^n \rvert) : n \in \omega \rangle\) as a set \(A \subseteq \lambda\) where \(\lambda\) is a cardinal, as in the proof of Corollary \ref{MantleTheorem}. Take a singular strong limit cardinal \(\kappa > \lambda\) with \(\text{cf}(\kappa) > \omega\). \newline

\noindent In \(V\) and in \(V[G]\), define a sequence \(\langle S_{\alpha} : \alpha \in ON \rangle\) of binary relations on the set \newline
\(B(\kappa) := \{x \subseteq \kappa : \forall n \in \omega \thinspace [\mathbb{P}^n \text{ preserves } HOD_x]\}\) by recursion: \newline
\(xS_0y\) iff \(y \in HOD_x\); \newline
\(xS_{\alpha + 1}y\) iff \(xS_{\alpha}y\), \(y \in \text{dom}(S_{\alpha})\), and \(\exists d \in \omega^{\omega} \cap HOD_x \thinspace [d \text{ dominates } HOD_y]\); \newline
and for limit \(\gamma\), \(xS_{\gamma}y\) iff \(\forall \alpha < \gamma \thinspace [xS_{\alpha}y]\). \newline

\noindent In \(V\) and in \(V[G]\): \(\langle S_{\alpha} : \alpha \in ON \rangle\) is decreasing, so it eventually stabilizes by Replacement. Write \(S_{\infty}\) for the limit relation.

\begin{Claim}
In \(V\) and in \(V[G]\): \(S_{\infty} = \emptyset\).
\end{Claim}

\noindent \textit{Proof of Claim.} For each \(x \subseteq \kappa\), let \(<_x\) denote the usual well-order of \(HOD_x\) that is definable from \(x\). Note that \(\forall x \in \text{dom}(S_{\infty}) \thinspace \exists y \in \text{dom}(S_{\infty}) \cap HOD_x \thinspace \exists d \in \omega^{\omega} \cap HOD_x \thinspace [d \text{ dominates } HOD_y]\). \newline

\noindent Toward a contradiction, assume \(S_{\infty} \neq \emptyset\). Take \(x_0 \in \text{dom}(S_{\infty})\). \newline

\noindent Define \(\langle x_n, d_n : n \in \omega \rangle\) by recursion: given \(x_n \in \text{dom}(S_{\infty})\), \newline
let \((x_{n+1}, d_n) \in (\mathcal{P}(\kappa) \times \omega^{\omega}) \cap HOD_{x_n}\) be the \(<_{x_n}\)-least pair such that \newline
\(x_{n+1} \in \text{dom}(S_{\infty})\) and \(d_n\) dominates \(HOD_{x_{n+1}}\). \newline

\noindent In the end, \(\forall n \in \omega \thinspace [HOD_{x_n} \supseteq HOD_{x_{n+1}}]\) and \(\forall n \in \omega \thinspace [d_n \text{ dominates } HOD_{x_{n+1}}]\). \newline
Also, for each \(m \in \omega\): \(\langle x_n, d_n : n \geq m \rangle\) is hereditarily definable from ordinals and \(x_m\), \newline
i.e. \(\langle x_n, d_n : n \geq m \rangle \in HOD_{x_m}\). This contradicts Fact \ref{Hjorth}. \hfill \qedsymbol \newline

\noindent In \(V[G]\), by Fact \ref{Shelah}, take \(x \subseteq \kappa\) so that \(\mathcal{P}(\kappa) \subseteq HOD_x\) (so we also have \(\omega^{\omega} \subseteq HOD_x\)) and \(x\) enumerates \(A \subseteq \lambda\) in its first \(\lambda\) bits. Then in any generic extension of \(V[G]\), we can recover \(a = \langle \mathcal{P}(\lvert \mathbb{P}^n \rvert)^{V[G]} : n \in \omega \rangle\) from ordinals and \(x\). So by Lemma \ref{PreservesLemma}, we have \(x \in B(\kappa)^{V[G]}\). \newline

\noindent Next, we have \(B(\kappa)^V \subseteq B(\kappa)^{V[G]}\) (because if \(V \models ``\mathbb{P}^{n + 1} \text{ preserves } HOD_y\)", then \newline
\(V[G] \models ``\mathbb{P}^n \text{ preserves } HOD_y\)") and \(\forall y \in B(\kappa)^V \thinspace [HOD_y^V = HOD_y^{V[G]}]\). So by induction, we have \(\forall \alpha \in ON \thinspace [S_{\alpha}^V \subseteq S_{\alpha}^{V[G]}]\); so by induction (since \(\mathcal{P}(\kappa)^V \subseteq HOD_x^{V[G]}\)), we have \(\forall \alpha \in ON \thinspace \forall y \in \text{dom}(S_{\alpha}^V) \thinspace [xS_{\alpha}^{V[G]}y]\). \newline

\noindent Take \(d \in \omega^{\omega} \cap V[G]\) so that \(d\) dominates \(V\). We will contradict the claim by showing that for each \(\alpha \in ON\), we have \(x \in \text{dom}(S_{\alpha}^{V[G]})\). \newline

\noindent The base case and the limit case are clear. \newline
For the successor case, let \(\alpha \in ON\), and suppose \(x \in \text{dom}(S_{\alpha}^{V[G]})\). Since \(V \equiv_{ON} V[G]\), take \(y \in \text{dom}(S_{\alpha}^V)\). Then \(xS_{\alpha}^{V[G]}y\); \(y \in \text{dom}(S_{\alpha}^{V[G]})\); \(d \in \omega^{\omega} \cap HOD_x^{V[G]}\); and \(d\) dominates \(HOD_y^{V[G]}\). So \(xS_{\alpha+1}^{V[G]}y\), so \(x \in \text{dom}(S_{\alpha+1}^{V[G]})\).
\end{proof}

The proof becomes cleaner if \(\mathbb{P}\) is (for example) Hechler forcing: since Hechler forcing is always definable and cone homogeneous, the assumption \(V \equiv_{ON} V[G]\) is only needed at the very end, when comparing \(S_{\alpha}^V\) and \(S_{\alpha}^{V[G]}\). So if \(G\) is Hechler-generic over \(V\), then \(V\) and \(V[G]\) disagree on the first \(\alpha\) satisfying \(S_{\alpha} = \emptyset\).

\section{Open problems}

First, we ask two questions about theories. Recall that in Woodin's model from the introduction, we have \(V = L[r]\) for a real \(r\), and the theory is unchanged by any forcing whose generic object is a real.

\begin{Question}
Must the theory be changed by some forcing whose generic object is a subset of \(\omega_1\)?
\end{Question}

For example, in Woodin's model, the theory is changed by \(\text{Fn}(\omega_1, 2)\) (i.e. the set of finite partial functions from \(\omega_1\) to \(2\), ordered by \(\supseteq\)), because \(\text{Fn}(\omega_1, 2)\) forces ``there is no real \(r\) such that \(V = L[r]\)".

\begin{Question}
Suppose there is no real \(r\) such that \(V = L[r]\). Must Hechler forcing change the theory?
\end{Question}

We can ask further questions by strengthening the assumption, or by replacing ``Hechler forcing" with other forcings that add dominating reals.

Finally, we ask a question about ordinal theories.

\begin{Question}
Let \(\mathbb{P}\) be a forcing that adds a dominating real. Must \(\mathbb{P}\) change the ordinal theory?
\end{Question}

A positive answer would be a stronger result than both Corollary \ref{MantleTheorem} and Theorem \ref{GenTheorem}.

\bibliography{domreferences}
\bibliographystyle{plain}

\end{document}